\documentclass[reqno,a4paper]{amsart}
\usepackage{amsmath}
\usepackage{amssymb}
\usepackage{amsthm}
\usepackage{enumerate}
\usepackage[mathscr]{eucal}
\usepackage[colorlinks=true, citecolor=blue, linkcolor=blue]{hyperref}

\theoremstyle{plain}
\newtheorem{thm}{Theorem}[section]
\newtheorem{cor}[thm]{Corollary}
\newtheorem{lem}[thm]{Lemma}
\newtheorem{note}{Note}[section]

\newtheorem{exam}{Example}[section]
\theoremstyle{definition}
\newtheorem{defn}{Definition}[section]

\begin{document}

\setcounter {page}{1}
\title{ Generalizations of bornological convergence and convergence of partial maps via ideal}

\author[ P. Malik and A. Ghosh ]{Prasanta Malik* and Argha Ghosh* \ }
\newcommand{\acr}{\newline\indent}
\maketitle \noindent {\address{{*\,} Department of Mathematics,
The University of Burdwan, Burdwan-713104, West Bengal, India.
Email: pmjupm@yahoo.co.in, buagbu@yahoo.co.in }} \maketitle
\begin{abstract}
{ In this paper using the notion of an ideal $\mathcal I$ on a
directed set, we extend the notion of convergence of nets of
partial maps to the notions of $\mathcal I$-convergence ( or
filter convergence) of nets of partial maps and $\mathcal
I^*$-convergence of nets of partial maps.}
\end{abstract}
\author{}
\maketitle



\section{\textbf{Introduction}}

A partial map or a partial function from a metric space $(X,d)$ to
a metric space $(Y, \mu)$ is a pair $(F,\gamma)$ where $F$ is a
closed non-empty subset of $X$ and $\gamma$ is a function from $F$
to $Y$. The space of all partial maps from $X$ to $Y$ is denoted
by $\mathcal P[X, Y]$ and the space of all continuous partial maps
from $X$ to $Y$ is denoted by $\mathcal C[X, Y]$.

The notion of bornological convergence of nets was introduced by
Lechicki et al. \cite{Le} and the notion of convergence of partial
maps was introduced by Beer et al. \cite{Be}. To extended the
ideas of bornological convergence and convergence of partial maps,
in this paper using the notion of an ideal $\mathcal I$ on a
directed set, we first, generalize the notion of bornological
convergence of nets to bornological $\mathcal I$-convergence of
nets and then we generalize the concept of convergence of nets of
partial maps to the notion of $\mathcal I$-convergence of nets of
partial maps. Further we introduce the concept of bornological
$\mathcal I^*$-convergence of nets as well as the concept of
$\mathcal I^*$-convergence of nets of partial maps and study their
relationship with bornological $\mathcal I$-convergence of nets
and $\mathcal I$-convergence of nets of partial maps respectively.

\section{\textbf{Basic definitions and notation  }}
\begin{defn}\cite{Ko}
 If $X$ is a non-empty set, then a family $\mathcal I\subset 2^X$ is
called an ideal of $X$ if\\
(i) $\emptyset\in \mathcal I$,\\
(ii) $A, B \in\mathcal I$ implies $A \cup B \in\mathcal I$, and\\
(iii) $A \in \mathcal I$ ; $B\subset A $ implies $B \in\mathcal I$.\\\\
The ideal $\mathcal I$ is called non-trivial if $\mathcal
I\neq\left\{\emptyset\right\}$ and $X\notin\mathcal I$.
\end{defn}
\begin{defn}\cite{Ko}
A non empty family $\mathcal F$ of subsets of a non-empty set $X$ is called a filter of $X$ if\\
(i) $\emptyset\notin \mathcal F$,\\
(ii) $A, B\in \mathcal F$ implies $A \cap B \in\mathcal F$ and\\
(iii) $A \in\mathcal F$ ; $A \subset B$ implies $B \in\mathcal F$.\\\\
Clearly, if $\mathcal I \subset 2^X$ is a non-trivial ideal of
$X$, then
\begin{center}
 $\mathcal {F(I)} =\left\{A\subset X~:~ X\setminus A\in \mathcal I\right\}$
\end{center}
is a filter on $X$, called the filter associated with $\mathcal
I$.
\end{defn}

\noindent A non-trivial ideal of $X(\neq\emptyset)$ is called
admissible if $\{y\}\in\mathcal I$ for each $y\in X$.

Throughout the paper, $X = (X, d)$, $Y = (Y,\mu)$ will denote
metric spaces, $ \mathcal P[X, Y]$ will denote the space of all
partial maps from the metric space $X$ to the metric space $Y$,
$\mathcal G$ will denote a directed set $ (\mathcal G, \geq)$ and
$\mathcal I$ will denote a non-trivial ideal of $\mathcal G$
unless otherwise mentioned.

An ideal $\mathcal I$ of $\mathcal G$ will be written sometimes as
$\mathcal I_{\mathcal G}$ to indicate the directed set $\mathcal
G$ of which $\mathcal I$ is an ideal.

 Let $\mathcal M_\gamma = \left\{\alpha\in\mathcal G:\alpha\geq \gamma\right\}$, $\gamma \in\mathcal G$. Then the collection
\begin{center}
$\mathcal F_0 =\left\{A\subset\mathcal G:A\supset \mathcal
M_\gamma~for~some~\gamma \in \mathcal G \right\}$
\end{center}
forms a filter in $\mathcal G$. Let $\mathcal I_0
=\left\{A\subset\mathcal G:\mathcal G\setminus A\in\mathcal
F_0\right\}$. Then $\mathcal I_0$ is a non-trivial ideal in
$\mathcal G$.

\begin{defn}\cite{La}
A non-trivial ideal $\mathcal I$ of $\mathcal G$ is called
$\mathcal G$-$admissible $ if
$ \mathcal M_n \in{\mathcal F(\mathcal I)}$ for all $n \in \mathcal G$. \\
\end{defn}

\begin{defn}\cite{La2}\label{Se}
A net $\{x_\gamma\}_{\gamma\in\mathcal G}$ in a metric space $(X,
d)$ is said to be $\mathcal I$-convergent to $x\in X$ if for every
$\epsilon> 0 $, $\{\gamma\in \mathcal G: d(x_\gamma,x)\geq
\epsilon\}\in\mathcal I$. In this case, we write $\mathcal
I-lim~x_\gamma = x $.
\end{defn}
\begin{note}
The above definition of $\mathcal I$-convergence can be written as
follows: a net $\left\{x_\gamma\right\}_{\gamma \in\mathcal G}$ in
a metric space $(X,d)$ is said to be $\mathcal I$-convergent to
$x\in X$ if for every $\epsilon>0$,
\begin{center}
$\left\{\gamma\in\mathcal{G}:d(x_\gamma,x)<\epsilon\right\}\in
{\mathcal {F(I)}}.$
\end{center}
\end{note}

We now discuss the concept of bornology and the notion of
bornological convergence on a metric
space (for more details see \cite{Be, Le}).  \\

\noindent If $a \in X$ and $\epsilon > 0$, then $B(a, \epsilon)$
denotes the open $\epsilon$-ball with center at $a$ and radius
$\epsilon$. If $G$ is a nonempty subset of $X$, we write $d(a,G)$
to denote the distance from $a$ to $G$ and $G^\epsilon$ to denote
the $\epsilon$-enlargement of the set $G$, i.e.,
\begin{center}
$G^\epsilon=\left\{x:d(x,G)<\epsilon\right\}=\underset{x\in
G}{\bigcup}B(x,\epsilon)$.
\end{center}

\begin{defn}
If $X$ is a non-empty set, then a family $\mathcal B$ of non-empty
subsets of $X$ is
called a bornology on $X$ if the following conditions are satisfied:\\
(i) $\mathcal B$ covers $X$,\\
(ii) $A, B \in \mathcal B$ implies $A \cup B \in \mathcal B$ and\\
(iii) $A \in \mathcal B$ ; $B\subset A $ implies $B \in \mathcal
B$.

\end{defn}

\noindent The family $\mathcal F$ of all non-empty finite subsets
of $X$ is the smallest bornology on $X$ and the family $\mathcal
P_0(X)$ of all non-empty subsets of $X$ is the largest bornology
on $X$. Other important bornologies are: the family $\mathcal B_d$
of the nonempty $d$-bounded subsets, the family $\mathcal B_{tb}$
of the nonempty $d$-totally bounded subsets and the family
$\mathcal K$ of nonempty subsets
of $X$ whose closures are compact.\\

\begin{defn}\cite{Le}
Let $\mathcal B$ be a bornology on a metric space $(X,d)$. A
subfamily $\mathcal B'$ of $\mathcal B$ which is cofinal in
$\mathcal B$ with respect to inclusion is called a \texttt{base}
for the bornology $\mathcal B$. A bornology is called
\texttt{local} if it contains a small ball around each point of
$X$. Also, a bornology is said to be \texttt{stable under small
enlargements} if for every $B\in\mathcal B$ there is
$\varepsilon>0$ such that $B^\varepsilon\in \mathcal B$.
\end{defn}
\begin{defn}\cite{Le}
Let $(X, d)$ be a metric space and $\mathcal B$ be a bornology on
$(X,d)$. A net $\{D_\gamma\}_{\gamma\in \mathcal G}$ in $\mathcal
P_0(X)$ is called $\mathcal B^-$-convergent (lower bornological
convergent) to $D\in \mathcal P_0(X)$ if for every $B\in\mathcal
B$ and $\epsilon>0$, the following condition occurs eventually:
\begin{center}
 $D\cap B\subset D_\gamma^\epsilon.$
\end{center}
In this case, we write $D\in\mathcal B^--lim ~D_\gamma$.

Similarly the net is called $\mathcal B^+$-convergent (upper
bornological convergent) to $D\in\mathcal P_0(x)$ if for every
$B\in\mathcal B$ and $\epsilon>0$, the following condition occurs
eventually:
\begin{center}
 $D_\gamma\cap B\subset D^\epsilon$.
\end{center}
  In this case, we write $D\in\mathcal B^+-lim ~D_\gamma$.
\end{defn}

\noindent When both the upper and lower bornological convergences
occur, we say two-sided bornological convergence occurs and we
write $D\in\mathcal B-lim ~D_\gamma$.

\begin{defn}\cite{Be}
Let $(X,d)$, $(Y,\mu)$ be metric spaces, $\mathcal B$ be a
bornology on $X$. A net
$\{(D_\gamma,u_\gamma)\}_{\gamma\in\mathcal G}$ in $\mathcal
P[X,Y]$ is said to be $\mathcal P^-(\mathcal B)$-convergent to
$(D,u)\in\mathcal P[X,Y]$, if for every $B\in \mathcal B$ and
$\epsilon>0$, the following inclusion holds eventually:
\begin{center}
$u(D\cap B_1)\subset[u_\gamma(D_\gamma\cap
B_1^\epsilon)]^\epsilon$, for all $B_1 \subset B$.
\end{center}
In this case we write $(D,u)\in\mathcal P^-(\mathcal B)-lim(D_\gamma,u_\gamma)$.\\

\noindent Similarly, the net
$\{(D_\gamma,u_\gamma)\}_{\gamma\in\mathcal G}$ in $\mathcal
P[X,Y]$ is said to be $\mathcal P^+(\mathcal B)$-convergent to
$(D,u)\in\mathcal P[X,Y]$, if for every $B\in \mathcal B$ and
$\epsilon>0$, the following inclusion holds eventually:
\begin{center}
$u_\gamma(D_\gamma\cap B_1)\subset[u(D\cap
B_1^\epsilon)]^\epsilon$, for all $B_1 \subset B$.
\end{center}
In this case we write $(D,u)\in\mathcal P^+(\mathcal
B)-lim(D_\gamma,u_\gamma)$.
\end{defn}

\noindent If the net $\{(D_\gamma,u_\gamma)\}_{\gamma\in\mathcal
G}$ is both $\mathcal P^-(\mathcal B)$-convergent and $\mathcal
P^+(\mathcal B)$-convergent to $(D,u)$, then we say that
$\{(D_\gamma,u_\gamma)\}_{\gamma\in\mathcal G}$ is $\mathcal
P(\mathcal B)$-convergent to $(D,u)$ and in this case, we write
$(D,u)\in\mathcal P(\mathcal B)-lim(D_\gamma,u_\gamma)$.

\begin{defn}\cite{Be}
Let $\mathcal B$ be a bornology on $(X,d)$ and $(E,v)\in\mathcal
P[X,Y]$. We say that $(E,v)$ is uniformly continuous relative to
the bornology $\mathcal B$ if for every $B\in\mathcal B$ with
$E\cap B\neq \phi$, the map
\begin{center}
$v:E\cap B\rightarrow Y$
\end{center}
 is uniformly continuous.

We say that $(E,v)$ is  strongly uniformly continuous relative to
the bornology $\mathcal B$ if for every $B\in\mathcal B$ and for
every $\epsilon>0$ there is $\eta>0$ such that
$\mu(v(x),v(w))<\epsilon$ whenever $d(x,w)<\eta$ and $x,w\in E\cap
B^\eta$.
\end{defn}

\section{\textbf{bornological $\mathcal I$-convergence and $\mathcal I$-convergence of nets of partial maps }}

\begin{defn}\label{Bi}
Let $\mathcal B$ be a bornology on $X$. A net
$\{D_\gamma\}_{\gamma\in \mathcal G}$ in $\mathcal P_0(X)$ is
called $\mathcal B_{\mathcal I}^-$-convergent (lower bornological
$\mathcal I$-convergent) to $D\in \mathcal P_0(X)$ if for every
$B\in\mathcal B$ and $\epsilon>0$,
\begin{center}
 $\{\gamma\in\mathcal G:~D\cap B\subset D_\gamma^\epsilon\}\in{\mathcal F(\mathcal I)}.$
\end{center}
In this case, we write $D\in\mathcal B_{\mathcal I}^--lim ~D_\gamma$.\\
\noindent Similarly, the net $\{D_\gamma\}_{\gamma\in \mathcal G}$
in $\mathcal P_0(X)$ is called $\mathcal B_{\mathcal
I}^+$-convergent (upper bornological $\mathcal I$-convergent) to
$D\in\mathcal P_0(X)$ if for every $B\in\mathcal B$ and
$\epsilon>0$,
\begin{center}
 $\{\gamma\in\mathcal G:~D_\gamma\cap B\subset D^\epsilon\}\in{\mathcal F(\mathcal I)}.$
\end{center}
In this case, we write $D\in\mathcal B_{\mathcal I}^+-lim
~D_\gamma$.
\end{defn}
\noindent If $D\in\mathcal B_{\mathcal I}^--lim ~D_\gamma$ as well
as $D\in\mathcal B_{\mathcal I}^+-lim ~D_\gamma$, then we say that
the net $\{D_\gamma\}_{\gamma\in \mathcal G}$ is bornological
$\mathcal I$-convergent to $D$ and in this case, we write
$D\in\mathcal B_{\mathcal I}-lim ~D_\gamma$.

\begin{defn}
Let $\mathcal B$ be a bornology on $X$. A net
$\{(D_\gamma,u_\gamma)\}_{\gamma\in\mathcal G}$ in $\mathcal
P[X,Y]$ is said to be $\mathcal {P_I}^-(\mathcal B)$-convergent to
$(D,u)\in\mathcal P[X,Y]$ if for every $B\in \mathcal B$ and
$\epsilon>0$,
\begin{center}
$\{\gamma\in\mathcal G:\forall B_1(\subset B),~u(D\cap
B_1)\subset[u_\gamma(D_\gamma\cap
B_1^\epsilon)]^\epsilon\}\in{\mathcal F(\mathcal I)}.$
\end{center}
 In this case, we write $(D,u)\in\mathcal {P_I}^-(\mathcal B)-lim(D_\gamma,u_\gamma)$.

\noindent Similarly, the net
$\{(D_\gamma,u_\gamma)\}_{\gamma\in\mathcal G}$ in $\mathcal
P[X,Y]$  is said to be $\mathcal {P_I}^+(\mathcal B)$-convergent
to $(D,u)\in \mathcal P[X,Y]$ if for every $B\in \mathcal B$ and
$\epsilon>0$,
\begin{center}
$\{\gamma\in\mathcal G:\forall B_1(\subset
B),~u_\gamma(D_\gamma\cap B_1)\subset[u(D\cap
B_1^\epsilon)]^\epsilon\}\in{\mathcal {F(I)}}$.
\end{center}
In this case, we write $(D,u)\in\mathcal {P_I}^+(\mathcal
B)-lim(D_\gamma,u_\gamma)$.
\end{defn}
\noindent If $(D,u)\in\mathcal {P_I}^-(\mathcal
B)-lim(D_\gamma,u_\gamma)$ as well as $(D,u)\in\mathcal
{P_I}^+(\mathcal B)-lim(D_\gamma,u_\gamma)$, then we say that the
net $\{(D_\gamma,u_\gamma)\}_{\gamma\in\mathcal G}$ is $\mathcal
{P_I}(\mathcal B)$-convergent to $(D,u)$ and we denote it by
$(D,u)\in\mathcal {P_I}(\mathcal B)-lim(D_\gamma,u_\gamma)$.

\noindent We now cite an example of a net of partial maps which is
$\mathcal {P_I}^+(\mathcal B)$-convergent but not $\mathcal
P^+(\mathcal B)$-convergent.
\begin{exam}
Let $X=[-1,0]$ and $Y=\mathbb R$ be two metric spaces with usual
metric. Let $\mathcal I = \mathcal I_d=\{A\in\mathbb N: d(A)=0\}$,
where $d(A)$ is the asymptotic density of the set $A$, defined by
$d(A)=\underset{n\rightarrow
\infty}{\lim}\frac{\left|A(n)\right|}{n}$, where $A(n)=\left\{j\in
A:j\leq n\right\}$ and $\left|A(n)\right|$ represents the number
of elements in $A(n)$. Then $\mathcal I_d$ is a non-trivial ideal
in $\mathbb N$. Let $A=\{k^3:k\in\mathbb N\}$. Then $A$ is an
infinite subset of $\mathbb N$ and $A\in \mathcal I_d$. Let us
consider the function $u:X\rightarrow Y$ defined by $u(x)=0$,
$\forall x\in X$. Now for every $n\in A$, let us define $u_n:
X\rightarrow Y$ by
\[ u_n(x) = \left\{
  \begin{array}{l l}
    -nx-1, & \quad \text{if $-\frac{1}{n}\leq x$ }\\
    u(x), & \quad \text{$otherwise$}.
  \end{array} \right.\]\\
and for every $n\notin A$, we define, $u_n: X\rightarrow Y$  by
$u_n(x)=u(x)$, $\forall x\in X$. Then with respect to the
bornology $\mathcal B=\mathcal P_0(X)$, the sequence of partial
maps $ \{(X, u_n)\}_{n\in\mathbb N}$ is not $\mathcal P^+(\mathcal
B)$-convergent to $(X, u)$ but is $\mathcal P_{\mathcal
I_{d}}^+(\mathcal B)$-convergent to $(X, u)$.
\end{exam}

\begin{thm}\label{T1}
Let $\{(D_\gamma,u_\gamma)\}_{\gamma\in\mathcal G}$ be a net in
$\mathcal P[X,Y]$ and $\mathcal B$ be a bornology on $X$.
\begin{center}
$(a)$ If $(D,u)\in \mathcal {P_I}^-(\mathcal
B)-lim~(D_\gamma,u_\gamma)$, then $\forall B\in \mathcal B$ and
$\forall\epsilon>0$,  $\{\gamma\in\mathcal G:~D\cap B\subset
D_\gamma^\epsilon\}\in{\mathcal F(\mathcal I)}.$
\end{center}
\begin{center}
$(b)$ If $(D,u)\in \mathcal {P_I}^+(\mathcal
B)-lim~(D_\gamma,u_\gamma)$, then $\forall B\in \mathcal B$ and
$\forall\epsilon>0$,  $\{\gamma\in\mathcal G:~D_\gamma\cap
B\subset D^\epsilon\}\in{\mathcal F(\mathcal I)}.$
\end{center}
\end{thm}
\begin{proof}
The proof directly follows from the definitions.
\end{proof}

\begin{thm}\label{T2}
 Let $\{(D_\gamma,u_\gamma)\}_{\gamma\in\mathcal G}$ be a net in $\mathcal P[X,Y]$ and $\mathcal B$ be a bornology on $X$. Then
\begin{center}
$(a)$ $(D,u)\in \mathcal {P_I}^+(\mathcal
B)-lim~(D_\gamma,u_\gamma)$ if and only if for every $B\in
\mathcal B$ and $\epsilon>0$,
\begin{center}
$\{\gamma\in\mathcal G:\underset{z\in D_\gamma\cap
B}{\sup}~~\underset{x\in D \cap
B_d(z,\epsilon)}{\inf}\mu(u(x),u_\gamma(z))<\epsilon\}\in{\mathcal
{F(I)}}.$
\end{center}
\end{center}
\begin{center}
$(b)$ $(D,u)\in \mathcal {P_I}^-(\mathcal
B)-lim~(D_\gamma,u_\gamma)$ if and only if for every $B\in
\mathcal B$ and $\epsilon>0$,
\begin{center}
$\{\gamma\in\mathcal G:\underset{z\in D\cap
B}{\sup}~~\underset{x\in D_\gamma\cap
B_d(z,\epsilon)}{\inf}\mu(u(z),u_\gamma(x))<\epsilon\}\in{\mathcal
{F(I)}}.$
\end{center}
\end{center}

\end{thm}
\begin{proof}
The proof is straight forward so is omitted.
\end{proof}

Now if $\mathcal B$ is a bornology on $X$, then $\{B\times Y:
B\in\mathcal B\}$ forms a base for some bornology $\mathcal B^*$
(say) on $X\times Y$. Using this bornology $\mathcal {B^*}$ on
$X\times Y$, we now show that the $\mathcal {P_I}^-(\mathcal B)$
and $\mathcal {P_I}^+(\mathcal B)$ convergences in $\mathcal
P[X,Y]$ are actually the lower and the upper bornological
$\mathcal I$-convergences of graphs respectively in $X\times Y$,
which extend the results of \cite{Le} as well as of \cite{Be}.

\begin{thm}
Let $\{(D_\gamma,u_\gamma)\}_{\gamma\in\mathcal G}$ be a net in
$\mathcal P[X,Y]$ and $\mathcal B$ be a bornology on $X$. Then,
for $(D,u)\in \mathcal P[X,Y]$, the following equivalences hold:
\begin{center}
$(a)$ $ Gr(u)\in(\mathcal {B^*_I})^--lim~Gr(u_\gamma)$ if and only
if $(D,u)\in\mathcal {P_I}^-(\mathcal B)-lim(D_\gamma,u_\gamma)$;
\end{center}
\begin{center}
$(b)$ $ Gr(u)\in(\mathcal {B^*_I})^+-lim~Gr(u_\gamma)$ if and only
if $(D,u)\in\mathcal {P_I}^+(\mathcal B)-lim(D_\gamma,u_\gamma)$.
\end{center}
\end{thm}
\begin{proof}
$(a)$ Let $(D,u)\in\mathcal {P_I}^-(\mathcal
B)-lim(D_\gamma,u_\gamma)$. To verify bornological convergence of
graphs, it is suffices to work with the basic sets in $\mathcal
B^*$. Let $B\times Y\in \mathcal B^*$, where $B\in\mathcal B$. Let
$\epsilon>0$ be given. Then by Theorem \ref{T1} and Theorem
\ref{T2} we have
\begin{center}
$A=\{\gamma\in \mathcal G:D\cap B\subset D_\gamma^\epsilon\}\in
{\mathcal {F(I)}}$ and
\end{center}
\begin{center}
$C=\{\gamma\in\mathcal G:\underset{z\in D\cap
B}{\sup}~~\underset{x\in D_\gamma\cap
B_d(z,\epsilon)}{\inf}\mu(u(z),u_\gamma(x))<\epsilon\}\in{\mathcal
{F(I)}}$.
\end{center}
Now $A, ~C\in{\mathcal {F(I)}}\Rightarrow A\cap C\in{\mathcal
{F(I)}}.$ Let $\gamma\in A\cap C$. Then both
\begin{equation}\label{eq1}
 D\cap B\subset D_\gamma^\epsilon
\end{equation}
and
\begin{equation}
\underset{z\in D\cap B}{\sup}~~\underset{x\in D_\gamma\cap
B_d(z,\epsilon)}{\inf}\mu(u(z),u_\gamma(x))<\epsilon \label{eq2}
\end{equation}
hold. Now fix $(z,u(z))\in (B\times Y)\cap Gr(u)$ so that $z\in
D$. Then by (\ref{eq1}) $B_d(z,\epsilon)\cap D_\gamma\neq \phi$
and by (\ref{eq2}) there exists some $x\in B_d(z,\epsilon)\cap
D_\gamma$, so that $\mu(u_\gamma(x),u(z))<\epsilon.$ So we have
$(x,u_\gamma(x))\in Gr(u_\gamma)$ and $(d\times
\mu)((z,u(z)),(x,u_\gamma(x)))<\epsilon$. This gives $Gr(u)\cap
(B\times Y)\subset Gr(u_\gamma)^\epsilon$. Therefore
\begin{center}
$A\cap C\subset \{\gamma\in\mathcal G:~ Gr(u)\cap(B\times
Y)\subset Gr(u_\gamma)^\epsilon\}.$
\end{center}
Since $A\cap C \in{\mathcal {F(I)}}$, we have $\{\gamma\in\mathcal
G:~ Gr(u)\cap(B\times Y)\subset Gr(u_\gamma)^\epsilon \}\in
{\mathcal {F(I)}}$.
Hence $ Gr(u)\in(\mathcal {B^*_I})^--lim~Gr(u_\gamma)$.\\

Conversely, let $ Gr(u)\in(\mathcal {B^*_I})^--lim~Gr(u_\gamma)$.
Let $B\in\mathcal B$ and $\epsilon>0$ be given. Choosing
$0<\eta<\epsilon$, we have
\begin{center}
$A_1=\{\gamma\in\mathcal G:~ Gr(u)\cap(B\times Y)\subset
Gr(u_\gamma)^\eta \}\in {\mathcal {F(I)}}.$
\end{center}
Let $\gamma\in A_1$ and $z\in D\cap B$ be arbitrary. Then
$(z,u(z))\in (B\times Y)\cap Gr(u)$. So there exists
$(y_0,u_\gamma(y_0))\in Gr(u_\gamma)$ such that
\begin{center}
$(d\times\mu)((z,u(z)),(y_0,u_\gamma(y_0)))<\eta$.
\end{center}
Thus we get $y_0\in D_\gamma$ such that $d(z,y_0)<\eta<\epsilon$
as well as $\mu(u(z),u_\gamma(y_0))<\eta$. Therefore
\begin{center}
$\underset{x\in B_d(z,\epsilon)\cap
D_\gamma}{\inf}\mu(u(z),u_\gamma(x))<\eta$.
\end{center}
\noindent Since $z\in D\cap B$ is arbitrary, we have
\begin{center}
$\underset{z\in D\cap B}{\sup}~~\underset{x\in B_d(z,\epsilon)\cap
D_\gamma}{\inf}\mu(u(z),u_\gamma(x))\leq\eta<\epsilon.$
\end{center}
Thus
\begin{center}
$A_1\subset\{\gamma\in\mathcal G:\underset{z\in D\cap
B}{\sup}~~\underset{x\in B_d(z,\epsilon)\cap
D_\gamma}{\inf}\mu(u(z),u_\gamma(x))<\epsilon\}.$
\end{center}
Since $A_1\in{\mathcal {F(I)}}$, we have $\{\gamma\in\mathcal
G:\underset{z\in D\cap B}{\sup}~~\underset{x\in
B_d(z,\epsilon)\cap
D_\gamma}{\inf}\mu(u(z),u_\gamma(x))<\epsilon\}\in {\mathcal
{F(I)}}$.
Then by Theorem \ref{T2} $(b)$ we have $(D,u)\in\mathcal {P_I}^-(\mathcal B)-lim(D_\gamma,u_\gamma)$. This completes the proof.\\
\noindent $(b)$ The proof of $(b)$ is similar to that of $(a)$, so
it is omitted.

\end{proof}

\begin{cor}
Let $\{(D_\gamma,u_\gamma)\}_{\gamma\in\mathcal G}$ be a net in
$\mathcal P[X,Y]$ and $\mathcal B$ be a bornology on $X$. Then
$\mathcal {P_I}(\mathcal B)$-convergence on $\mathcal P[X,Y]$
coincide with $\mathcal{B^*_I}$-convergence of graphs of partial
maps in $X\times Y$.
\end{cor}
\begin{thm}\label{T4}
Let $\{(D_\gamma,u_\gamma)\}_{\gamma\in\mathcal G}$ be a net in
$\mathcal P[X,Y]$ and $\mathcal B$ be a local bornology on $X$. If
the net $\{(D_\gamma,u_\gamma)\}_{\gamma\in\mathcal G}$ is
$\mathcal{P_I}(\mathcal B)$-convergent to both the partial maps
$(S,u)$ and $(T,v)$, then $S=T$.
\end{thm}
\begin{proof}
Let the net $\{(D_\gamma,u_\gamma)\}_{\gamma\in\mathcal G}$ be
$\mathcal {P_I}(\mathcal B)$-convergent to both the partial maps
$(S,u)$ and $(T,v)$. If possible, let $S\neq T$. Then without any
loss of generality we assume $x\in S\setminus T$. Since $T$ is a
closed subset of $X$, we have $\{x\}^\epsilon\cap T=\emptyset$ for
some $\epsilon>0$. Choose $0<\delta<\frac{\epsilon}{2}$ so that
$\{x\}^\delta\in \mathcal B$. Let $B=\{x\}^\delta$. Since the net
is $\mathcal {P_I}^-(\mathcal B)$-convergent to $(S,u)$, we have
\begin{center}
$\{ \gamma\in\mathcal G: \forall B_1(\subset B), u(S\cap
B_1)\subset [u_\gamma(D_\gamma\cap B_1^\delta)]^\delta\}\in
\mathcal {F(I)}$.
\end{center}
Let us take $B_1=\{x\}$. Then we have
\begin{center}
$P=\{ \gamma\in\mathcal G:  u(S\cap \{x\})\subset
[u_\gamma(D_\gamma\cap \{x\}^\delta)]^\delta\}\in \mathcal
{F(I)}$.
\end{center}
Since $u(S\cap\{x\})\neq\emptyset$, we have $u_\gamma(D\gamma\cap
\{x\}^\delta)\neq \emptyset$ for all $\gamma\in P$. Now
\begin{center}
$v(T\cap B^{\frac{\epsilon}{2}})\subset v(T\cap
\{x\}^\epsilon)=\emptyset$.
\end{center}
Thus for all $\gamma\in P$ we have
\begin{center}
$ u_\gamma(D_\gamma\cap B)\not\subset {[v(T\cap
B^{\frac{\epsilon}{2}})]^{\frac{\epsilon}{2}}} .$
\end{center}
Hence
\begin{center}
$P\subset \{\gamma\in\mathcal G: u_\gamma(D_\gamma\cap
B)\not\subset {[v(T\cap
B^{\frac{\epsilon}{2}})]^{\frac{\epsilon}{2}}} \} $.
\end{center}
Since $P\in \mathcal {F(I)}$, we have
 \begin{center}
$ \{\gamma\in\mathcal G: u_\gamma(D_\gamma\cap B)\not\subset
{[v(T\cap B^{\frac{\epsilon}{2}})]^{\frac{\epsilon}{2}}} \}\in
\mathcal {F(I)} $.
\end{center}
Thus
\begin{center}
$ \{\gamma\in\mathcal G: u_\gamma(D_\gamma\cap B)\subset {[v(T\cap
B^{\frac{\epsilon}{2}})]^{\frac{\epsilon}{2}}} \}\notin\mathcal
{F(I)} $,
\end{center}
which is a contradiction, since
$\{(D_\gamma,u_\gamma)\}_{\gamma\in\mathcal G}$ is $\mathcal {
P_{I}(B)} $-convergent to $(T, v)$. Therefore $S=T$.
\end{proof}

\section{\textbf{bornological $\mathcal I^*$-Convergence and $\mathcal I^*$-Convergence of nets of partial maps }}
In this section, we introduce the notions of bornological
$\mathcal I^*$-convergence and $\mathcal I^*$-convergence of nets
of partial maps and study their relationship with bornological
$\mathcal I$-convergence and $\mathcal I$-convergence of nets of
partial maps.

\begin{defn}
Let $(X,d)$ be a metric space and $ \mathcal B $ be a bornology on
$X$. A net $ \{D_\gamma\}_{\gamma\in \mathcal G} \in \mathcal P_0
(X) $ is said to be $\mathcal B_{\mathcal I^*}$-convergent
(bornological $\mathcal I^*$ convergent) to $D \in \mathcal P_0
(X)$ if there exists a set $\mathcal G'\in \mathcal {F(I)}$ such
that $\mathcal G'$ itself is a directed set with respect to the
binary relation induced from $(\mathcal G, \geq)$ and the net
$\{D_\gamma\}_{\gamma\in \mathcal G'}$ is $\mathcal B$-convergent
to $D$.

In this case, we write $D\in\mathcal B_{\mathcal I^*}-lim
D_\gamma$.
\end{defn}

\begin{thm}
Let $\mathcal I$ be a $\mathcal G$-$admissible $ ideal of a
directed set $(\mathcal G, \geq)$, $\{D_\gamma\}_{\gamma\in
\mathcal G}$ be a net in $\mathcal P_0 (X)$ and $ D \in \mathcal
P_0 (X)$. Then $D\in\mathcal B_{\mathcal I^*}-lim D_\gamma$
implies $D\in\mathcal B_{\mathcal I}-lim D_\gamma$.
\end{thm}

\begin{proof}
Let $D\in\mathcal B_{\mathcal I^*}-lim D_\gamma$. Then there
exists $\mathcal G'\in \mathcal {F(I)}$ such that $\mathcal G'$
itself is a directed set with respect to the binary relation
induced from $(\mathcal G, \geq)$ and $\{D_\gamma\}_{\gamma\in
\mathcal G'}$ is $\mathcal B$-convergent to $D$.

Let $B\in\mathcal B$ and $\epsilon>0$ be given. Then there exists
$\gamma_0\in \mathcal G'$, such that for all $\gamma\in \mathcal
G'$ with $\gamma\geq\gamma_0$, we have
\begin{center}
$D\cap B\subset D^\epsilon_\gamma ~ and ~ D_\gamma\cap D\subset
D^\epsilon$.
\end{center}
Since $\mathcal I$ is $\mathcal G$-admissible, $ \mathcal
M_{\gamma_0} \in \mathcal {F(I)}$. Then $\mathcal
M_{\gamma_0}\cap\mathcal G'\in \mathcal {F(I)}$. Now
\begin{center}
$\mathcal M_{\gamma_0}\cap\mathcal G'\subset\{\gamma\in\mathcal G:
D\cap B\subset D^\epsilon_\gamma\} \cap  \{\gamma\in\mathcal G:
D_\gamma\cap B\subset D^\epsilon\}$.
\end{center}
So $\{\gamma\in\mathcal G: D\cap B\subset
D^\epsilon_\gamma\}\in\mathcal {F(I)}$  as well as
$\{\gamma\in\mathcal G: D_\gamma\cap B\subset
D^\epsilon\}\in\mathcal {F(I)}$. Hence $D\in\mathcal B_{\mathcal
I}-lim D_\gamma$.
\end{proof}

\begin{defn}\cite{La}
Let $\mathcal I$ be a $\mathcal G$-$admissible $ ideal of a
directed set $\mathcal G$. Then $\mathcal I$ is said to satisfy
the condition (DP) if for every countable family of mutually
disjoint sets $\{P_1,P_2,...\}$ in $\mathcal I$ there exists a
countable family of sets $\{Q_1, Q_2,...\}$ in $\mathcal G$ such
that for each $i\in\mathbb N$, $P_i\Delta Q_i\subset \mathcal
G\setminus \mathcal M_{\gamma_i}$ for some $\gamma_i\in\mathcal G$
and $Q= \bigcup\limits_{ i= 1}^{\infty} Q_i\in\mathcal I$, where
$\Delta$ stands for the symmetric difference between two sets.
\end{defn}
\begin{lem}\label{La}
Let $\mathcal I$ be an ideal of a directed set $(\mathcal G,
\geq)$ satisfying the condition (DP). Then for any countable
family of sets $\{E_1,E_2,...\}$ in $\mathcal {F(I)}$ there exists
a $E\in\mathcal {F(I)}$ such that $E$ itself is a directed set
with respect to the binary relation induced from $(\mathcal G,
\geq)$ and for each $ i \in \mathbb{N}$, $E\setminus E_i \subset
\mathcal G\setminus \mathcal M_{\gamma^{(i)}}$ for some
$\gamma^{(i)}\in\mathcal G$.
\end{lem}

\begin{proof}
Let $\{E_1,E_2,...\}$  be a countable family of sets in $\mathcal
{F(I)}$. Then $\{F_1,F_2,...\}$ is a countable family of sets in
$\mathcal I$, where $F_i=\mathcal G\setminus E_i, \forall i \in
\mathbb{N}$. Now we construct a sequence of sets
$\{P_i\}_{i\in\mathbb N}$ as follows:
\begin{center}
$P_1=F_1$, $P_2=F_2\setminus F_1$,..., $P_i=F_i\setminus ( F_1\cup
F_2\cup...\cup F_{i-1}), \ldots$
\end{center}
Clearly, $P_i\in \mathcal I$, for all $i\in\mathbb N$ and $P_i
\cap P_j = \emptyset$ for $ i \neq j$. Since $\mathcal I$
satisfies the condition (DP), there exists a countable family of
sets $\{Q_1,Q_2,...\}$ in $\mathcal G$ such that for each
$i\in\mathbb N$, $P_i\Delta Q_i\subset \mathcal G \setminus
M_{\gamma_i}$ for some $\gamma_{i}\in\mathcal G$ and $Q=
\bigcup\limits_{ i =1 }^{ \infty} Q_i\in\mathcal I$. Now, fix $ i
\in \mathbb{N}$. Then
\begin{center}
$F_i\setminus Q\subset (\bigcup\limits_{j=1}^{i} F_j) \setminus Q
= (\bigcup\limits_{j=1}^{i} P_j) \setminus Q =
\bigcup\limits_{j=1}^{i}(P_j\setminus Q)
\subset\bigcup\limits_{j=1}^{i} (P_j\setminus
Q_j)\subset\bigcup\limits_{j=1}^{i} (P_j \Delta Q_j)
\subset\bigcup\limits_{j=1}^{i} (\mathcal G \setminus \mathcal M_{
\gamma _j}) $.
\end{center}
Now for $\gamma_1,\gamma_2,...,\gamma_i$, there exists
$\gamma^{(i)}\in \mathcal G$ such that $\gamma^{(i)}\geq
\gamma_j$, $\forall j=1,2,...,i$. Then
\begin{center}
$F_i\setminus Q\subset\bigcup\limits_{j=1}^{i} (\mathcal
G\setminus\mathcal M_{\gamma_j})\subset\mathcal G\setminus\mathcal
M_{\gamma^{(i)}} $.
\end{center}
Let $E=\mathcal G\setminus Q$. Then $E\in\mathcal {F(I)}$ and
$E\setminus E_i=F_i\setminus Q\subset G\setminus\mathcal
M_{\gamma^{(i)}} $.

We now show that $E$ itself is a directed set with respect to the
binary relation induced from $(\mathcal G, \geq)$. It is clear
that $\geq$ is reflexive and transitive on $E$. Now, let $e_1,
e_2\in E$. Since $e_1, e_2$  are two elements of $\mathcal G$,
there exists $e\in\mathcal G$ such that $e\geq e_1$ and $e\geq
e_2$. Now $E\cap\mathcal M_e\in\mathcal {F(I)} $. Let $e'\in
E\cap\mathcal M_e $. Then $e'\geq e$ and so $e'\geq e_1$ and
$e'\geq e_2$. Therefore $E$ is a directed set with respect to the
binary relation induced from $(\mathcal G, \geq)$.
\end{proof}

\begin{thm}
Let $\mathcal I$ be an ideal on a directed set $(\mathcal G,
\geq)$ satisfying the condition (DP), $(X, d)$ be a metric space
and $\mathcal B$ be a bornology on $X$. Then for any net
$\{D_\gamma\}_{\gamma\in \mathcal G}$ in $\mathcal P_0 (X)$ and $
D \in \mathcal P_0 (X)$, $D\in\mathcal B_{\mathcal I}-lim
D_\gamma$ implies $D\in\mathcal B_{\mathcal I^*}-lim D_\gamma$.
\end{thm}

\begin{proof}
Let $D\in\mathcal B_{\mathcal I}-lim D_\gamma$. Let $B\in\mathcal
B$ and $\epsilon>0$ be given. Since $D\in\mathcal B_{\mathcal
I}-lim D_\gamma$, for each $j \in \mathbb{N}$,
\begin{center}
$A_j=\{\gamma\in\mathcal G:D\cap B \subset
D^\frac{1}{j}_\gamma\}\in\mathcal {F(I)}$ and
\end{center}
\begin{center}
$B_j=\{\gamma\in\mathcal G:D_\gamma\cap B \subset
D^\frac{1}{j}\}\in\mathcal {F(I)}$.
\end{center}
Let $E_j = A_j \cap B_j, j \in \mathbb{N}$. Then $E_j \in\mathcal
{F(I)}, \forall j \in \mathbb{N}$. Since $\mathcal I$ satisfies
the condition (DP), by Lemma \ref{La}, there exists $E\in\mathcal
{F(I)}$ such that $E$ itself is a directed set with respect to the
binary relation induced from $(\mathcal G, \geq)$ and for each $ j
\in \mathbb{N}$, $E\setminus E_j \subset \mathcal G\setminus
\mathcal M_{\gamma^{(j)}}$ for some $\gamma^{(j)}\in\mathcal G$.
Now for the above $\epsilon>0$, there exists $j\in\mathbb N$ such
that $\frac{1}{j}<\epsilon$. Then for that $j$, $E\setminus E_j
\subset \mathcal G\setminus \mathcal M_{\gamma^{(j)}}$ for some
$\gamma^{(j)}\in\mathcal G$. Then for all $\gamma \in E$ and $
\gamma \geq \gamma^{(j)}$ we have $\gamma \in E_j = A_j \cap B_j$
and so
\begin{center}
$ D\cap B \subset D^\frac{1}{j}_\gamma \subset
D_{\gamma}^{\epsilon}$ and
\end{center}
\begin{center}
$D_\gamma\cap B \subset D^\frac{1}{j} \subset D^{\epsilon}$.
\end{center}
Therefore $D\in\mathcal B_{\mathcal I^*}-lim D_\gamma$.
\end{proof}
We now introduce the notion of $\mathcal I^*$-convergence of nets
of partial maps.

\begin{defn}
Let $(X,d), (Y, \mu)$ be two metric spaces, $ \mathcal B $ be a
bornology on $X$. A net $\{(D_\gamma, u_\gamma)\}_{\gamma\in
\mathcal G} \in \mathcal P[X, Y] $ is said to be $\mathcal
P_{\mathcal I^*}(\mathcal B)$-convergent to $(D, u) \in \mathcal
P[X, Y]$ if there exists a set $\mathcal G'\in \mathcal {F(I)}$
such that $\mathcal G'$ itself is a directed set with respect to
the binary relation induced from $(\mathcal G, \geq)$ and the net
$\{(D_\gamma, u_\gamma)\}_{\gamma\in \mathcal G'}$ is $\mathcal
P(\mathcal B)$-convergent to $(D, u)$.

In this case, we write $(D, u) \in\mathcal P_{\mathcal
I^*}(\mathcal B)-lim (D_\gamma, u_\gamma)$.
\end{defn}

\begin{thm}
Let $\mathcal I$ be a $\mathcal G$-$admissible $ ideal of a
directed set $(\mathcal G, \geq)$, $\{(D_\gamma,
u_\gamma)\}_{\gamma\in \mathcal G}$ be a net in $\mathcal P[X, Y]$
and $ ( D, u) \in \mathcal P[X, Y]$. Then $(D,u) \in\mathcal
P_{\mathcal I^*}(\mathcal B)-lim (D_\gamma, u_\gamma)$ implies
$(D,u)\in\mathcal P_{\mathcal I}(\mathcal B)-lim (D_\gamma,
u_\gamma)$.
\end{thm}

\begin{proof}
Starting from $(D,u) \in\mathcal P_{\mathcal I^*}(\mathcal B)-lim
(D_\gamma, u_\gamma)$ and using the proof technique of Theorem 4.1
we get $(D,u)\in\mathcal P_{\mathcal I}(\mathcal B)-lim (D_\gamma,
u_\gamma)$.
\end{proof}

\begin{thm}
Let $\mathcal I$ be an ideal on a directed set $(\mathcal G,
\geq)$ satisfying the condition (DP), $(X, d), (Y, \mu)$ be metric
spaces and $\mathcal B$ be a bornology on $X$. Then for any net
$\{(D_\gamma, u_\gamma)\}_{\gamma\in \mathcal G}$ in $\mathcal
P[X, Y]$ and $ (D, u) \in \mathcal P[X, Y]$, $(D, u)\in\mathcal
P_{\mathcal I}(\mathcal B)-lim (D_\gamma, u_\gamma)$ implies
$(D,u)\in\mathcal P_{\mathcal I^*}(\mathcal B)-lim (D_\gamma,
u_\gamma)$.
\end{thm}

\begin{proof}
Starting from $(D, u)\in\mathcal P_{\mathcal I}(\mathcal B)-lim
(D_\gamma, u_\gamma)$ and using the proof technique of Theorem 4.3
we get $(D,u)\in\mathcal P_{\mathcal I^*}(\mathcal B)-lim
(D_\gamma, u_\gamma)$.
\end{proof}

\noindent\textbf{Acknowledgment:} The second author is grateful to
the University Grants Commission, India for fellowship funding
under UGC-JRF (SRF fellowship) scheme during the preparation of
this paper.\\



\begin{thebibliography}{99}
\bibitem{Be} Beer, G., Caserta, A., Maio, G. Di, Lucchetti, R.: Convergence of Partial Maps. J. Math. Anal. Appl. 419, 1274-1289 (2014).
\bibitem{Fr} Fridy, J. A.: On statistical convergence. Analysis 5(4), 301-313 (1985).
\bibitem{Fa} Fast, H.: Sur la convergence statistique. Colloq. Math. 2, 241-244 (1951).
\bibitem{Ko} Caserta, A., Maio, G. Di, Ko\u{c}inac, L. D. R.: Statistical Convergence on Function spaces. Abstract and Applied Analysis (2011) DOI: 10.1155/2011/420419, 12 pages.
\bibitem{Ca} Caserta, A., Lucchetti, R.: Some Convergence Results for Partial Maps. Filomat 29(6), 1297-1305 (2015).
\bibitem{De} Debreu, G.: The Theory of Value: An Axiomatic Analysis of Economic Equilibrium. Yale University Press, New Heaven (1959).
\bibitem{Pa3} Das, P., Ghosal, S. Kr.: On $I$-Cauchy nets and completeness. Topology and its Applications 157, 1152-1156 (2010).
\bibitem{Ku} Kuratowski, K.: Sur l'espace des fonctions partielles. Ann. Mat. Pura Appl. 40, 61-67 (1955).
\bibitem{Ko} Kostyrko, P., Salat, T., Wilczynski, W.: $I$-convergence. Real Anal. Exchange 26(2), 669-685 (2000/2001).
\bibitem{Le} Lechicki, A., Levi, S., Spakowski, A.: Bornological Convergence. J. Math. Anal. Appl. 297, 751-770 (2004).
\bibitem{La} Lahiri, B. K., Das, P.: $I$ and $I^*$-Convergence of nets. Real Anal. Exchange 33(2), 431-442 (2007-2008).
\bibitem{La2} Lahiri, B. K., Das, P.: $ I $ and $ I^* $-convergence in topological spaces. Mathematica Bohemica 130(2), 153-160 (2005).
\bibitem{Sc} Schoenberg, I. J.: The integrability of certain
functions and related summability methods. Amer. Math. Monthly 66,
361-375 (1959).
\bibitem{Se} Sencimen, C., Pehilvan, S.: Strong ideal convergence in probabilistic metric spaces. Proc. Indian Acad. Sci. (Math. Sci.)119(3), 401-410 (2009).
\end{thebibliography}
\end{document}